\documentclass{amsart}

\usepackage{amsmath}
\usepackage{amssymb}
\usepackage{amsthm}

\usepackage{tikz}
\usetikzlibrary{matrix,arrows}

\usepackage[utf8]{inputenc}
\usepackage{cite}
\usepackage{hyperref}

\newcommand{\IZ}{\mathbb{Z}}

\newcommand{\IQ}{\mathbb{Q}}
\newcommand{\IC}{\mathbb{C}}
\newcommand{\IN}{\mathbb{N}}
\newcommand{\IF}[1]{\mathbb{F}_{#1}}

\newcommand{\image}{\operatorname{Im} }

\newcommand{\Hom}[2]{\operatorname{Hom}(#1, #2)}

\newcommand{\ICellC}{\operatorname{C}}

\newcommand{\Ho}[2]{H_{#1}(#2)}
\newcommand{\cHo}[2]{H^{#1}(#2)}

\newcommand{\ThEul}[1]{\chi_-(#1)}
\newcommand{\width}{\operatorname{width}}
\newcommand{\Thur}[1]{\left\| #1 \right\|_T}
\newcommand{\Map}[3]{#1 \colon #2 \rightarrow #3}

\newcommand{\IZfZ}[1]{\IZ/#1 \IZ}
\newcommand{\IZk}{\IZfZ{k}}
\newcommand{\IZp}{\IZfZ{p}}
\newcommand{\ind}[1]{\operatorname{ind}_{#1}}
\newcommand{\res}[1]{\operatorname{res}_{#1}}

\theoremstyle{plain}
\newtheorem{thm}{Theorem}[section]
\newtheorem*{thm*}{Theorem}

\newtheorem{lem}[thm]{Lemma}
\newtheorem{cor}[thm]{Corollary}

\theoremstyle{definition}
\newtheorem{defn}[thm]{Definition}

\newtheorem{exm}[thm]{Example}

\theoremstyle{remark}
\newtheorem{rem}[thm]{Remark}
\newtheorem*{claim}{Claim}

\begin{document}
\title[Twisted Reidemeister torsion and the Thurston norm]{Twisted Reidemeister torsion and the Thurston norm: graph manifolds and finite representations}
\author{Stefan Friedl}
\address{Department of Mathematics, University of Regensburg, Germany}
\email{sfriedl@gmail.com}
\author{Matthias Nagel}
\address{Department of Mathematics, University of Regensburg, Germany}
\email{matthias.nagel@mathematik.uni-regensburg.de}
\urladdr{http://homepages.uni-regensburg.de/~nam23094/}

\begin{abstract}
We show that the Thurston norm of any irreducible 3-manifold can be detected
using twisted Alexander polynomials corresponding to integral representations
and also corresponding to representations over finite fields. In particular our
result covers graph manifolds which is not covered by the earlier work of the
first author and Vidussi.
\end{abstract}

\maketitle
\section{Introduction}
Define for an oriented surface $\Sigma$ with 
components $\Sigma_i$ the complexity 
\begin{align*}
\chi_-(\Sigma) := \sum_i \max( -\chi(\Sigma_i), 0).
\end{align*}
For a $3$-manifold $N$ Thurston \cite{Thurston86} introduced a semi-norm on
$\Ho{2}{N, \partial N;\IZ}$. This semi-norm, now called \emph{Thurston norm},
is defined as
\begin{align*}
\Thur \sigma := \min \left\{ \ThEul{\Sigma} : \Sigma \text{ oriented, embedded surface with } [\Sigma] = \sigma \right\},
\end{align*}
where $[\Sigma]$ denotes the fundamental class of an oriented surface $\Sigma$.
By Poincar\'e duality we transfer this norm to $\cHo 1 {N;\IZ}$ and henceforth consider it only
as a semi-norm on cohomology.

We denote the cell complex of a universal cover of $N$ by $\ICellC(N)$. After
fixing a field $K$, we can twist this chain complex with a $(K(t),
\IZ[\pi_1(N)])$-bimodule $A$, obtaining $\ICellC(N;A) := A \otimes \ICellC(N)$.
For a $g \in \pi_1(N)$ the endomorphism $g_A \colon A \rightarrow A$ given by
$g_A(v) := v \cdot g$ is a linear map of the $K(t)$-vector space $A$.  
In Section \ref{sec:prelim} we recall the definition of 
the twisted Reidemeister torsion $\tau(N; A)\in K(t)$ of 
the chain complex $\ICellC(N;A)$. The Reidemeister
torsion $\tau(N;A)$ is a unit in $K(t)$ and well-defined up to multiplication with
$\pm \det g_A$ for a $g \in \pi_1(N)$. 

Given an element $p(t) \in K[t] \setminus \{0\}$ in the polynomial ring with $p(t) = \sum_{i=k}^{l}a_i t^i$ and
$a_l, a_k$ both non-zero, we define $\width p(t) := l-k$. We extend this
assignment to any-non zero element in the quotient field $K(t)$ by declaring
\begin{align*}
\width \left( p(t)/q(t) \right) &= \width p(t) - \width q(t),
\end{align*}
for non-zero $q(t) \in K[t]$. Finally, we use the convention that the width of the zero 
polynomial is zero.

\begin{defn}
\begin{enumerate}
\item A \emph{representation} $V$ of a group $G$ is a $(K, \IZ[G])$-bimodule, i.e.
a $K$-vector space with with a linear right action by $G$.
\item Let $V$ be a representation of $\pi_1(N)$.  For an element $\theta \in \cHo 1 {N;\IZ}$
the associated \emph{twisting-module} $V_\theta$ is the $(K(t), \IZ[\pi_1(N)])$-bimodule with
underlying $K(t)$-vector space $V_\theta := K(t) \otimes_K V$ and right
$\pi_1(N)$-action given by
\begin{align*}
(z \otimes v) \cdot g := z t^{\langle \theta, g\rangle} \otimes v \cdot g,
\end{align*}
where $\langle \theta, g\rangle$ is the evaluation of $\theta \in \cHo 1 {N;\IZ} = \Hom {\pi_1(N)} \IZ$
on $g \in \pi_1(N)$.
\end{enumerate}
\end{defn}

It has been known for a very long time that Reidemeister torsions, or perhaps
more precisely, its close cousin the Alexander polynomial, give a lower bound
on the genus of a knot.  This inequality was generalised in 
\cite[Theorem 1.1]{Friedl06} as follows.

\begin{thm}\label{thm:maintheorem-km}
Let $N$ be a $3$-manifold. Let 
$\theta \in \cHo 1 {N;\IZ}$ be a cohomology class.
For every representation $V$ of $\pi_1(N)$ the Thurston norm $\Thur \theta$ 
satisfies the inequality
\begin{align*}
\left( \dim V \right)\cdot \Thur \theta \geq \width \tau(N; V_\theta).
\end{align*}
\end{thm}

In \cite[Theorem 1.2]{Friedl12} it was shown that for any irreducible
3-manifold that is not a closed graph manifold there exists a unitary
representation such that the corresponding twisted Reidemeister torsions detect
the Thurston norm.
The proof of that result relies on the fact, as it was put in \cite{Agol15},
that  by the work of Agol \cite{Agol13}, Przytycki--Wise  \cite{PWise12} and
Wise \cite{Wise12} such 3-manifolds are `full of cubulated goodness'.

The following theorem is the main result of this paper.

\begin{thm}\label{thm:mainthm}
Let $N$ be an irreducible 3-manifold which is not $D^2 \times S^1$. 
For every $\theta \in \cHo{1}{N; \IZ}$ 
there is a representation $V$ factoring through a finite group 
which detects the Thurston norm of $\theta$, i.e.\ such that
\begin{align*}
\left( \dim V \right)\cdot \Thur \theta =\width \tau(N; V_\theta).
\end{align*}
Additionally, the representation $V$ can be chosen to be either
\begin{enumerate}
\item defined over the complex numbers and be integral, or
\item defined over a finite field $\IF q$ for almost all primes $q$.
\end{enumerate}
\end{thm}

Note that a complex representation which factors through a finite group
can be made unitary.

Our main theorem extends \cite[Theorem 1.2]{Friedl12} in two ways:
It extends the statement over to closed graph manifolds, that were excluded in
\cite{Friedl12} since these are in general `not full of cubulated goodness',
see \cite{Liu13}. This extension relies on recent work of the second author
\cite[Theorem~2.15]{Nagel14}.

Secondly, our theorem gives a refined statement about which types of representations can
detect the Thurston norm. In particular the result that representations over
finite fields can be used plays a critical role in the proof in \cite{Boileau15}
that the profinite completion of the knot group determines the knot genus.

We conclude this introduction with an observation. By \cite[Section~3]{Friedl06} the degrees of twisted Reidemeister torsions of a 3-manifold $N$ only depend on the fundamental group and on whether or not $N$ has boundary. We thus obtain the following corollary.

\begin{cor}\label{cor:samethurston}
The Thurston norm of an irreducible $3$-manifold is an invariant of its fundamental group.
\end{cor}

For closed 3-manifolds that is of course also a consequence of the fact that irreducible 3-manifolds that are not lens spaces are determined by their fundamental groups, see \cite[Chapter~2.1]{AFW15} for detailed references. For 3-manifolds with boundary the statement is slightly less obvious, since there are non-homeomorphic irreducible 3-manifolds with non-trivial boundary but isomorphic fundamental groups. It should not be hard though to prove Corollary \ref{cor:samethurston} using the theory of Dehn flips introduced by Johannson, see \cite[Section~29]{Johannson79} and \cite[Chapter~2.2]{AFW15} for details.

\subsection*{Conventions}
A $3$-manifold is understood to be connected, smooth, compact, orientable
and having only toroidal boundary, which can be empty. A vector space is also
understood to be finite dimensional. 
\subsection*{Acknowledgements}
The second author thanks Johannes Sprang for helpful discussions.

Both authors were supported 
SFB 1085 `Higher Invariants' funded by the Deutsche Forschungsgemeinschaft (DFG) at the University of Regensburg.
The second author was also supported by the DFG
in the GK 1269 at the University of Cologne.

\section{Preliminaries}\label{sec:prelim}
For this section we fix an irreducible $3$-manifold $N$ with a CW-structure.
Let $K$ be a field.

A universal cover $\Map \pi {\tilde N} N$ inherits an induced CW-structure.
The deck transformations act on $\tilde N$ from the left. 
With this left action the cellular chain complex of $\tilde N$ is a
chain complex of left $\IZ[\pi_1(N)]$-modules, which we denote by $\ICellC(N)$.
For a subcomplex $Y \subset N$ the preimage $\pi^{-1}(Y)$ is a CW-subcomplex
of $\tilde N$ and invariant under deck transformations. We define
$\ICellC(Y \subset N)$ to be the cellular complex of $\pi^{-1}(Y)$. This is a complex
of left $\IZ[\pi_1(N)]$-modules.
A lift of the cells of the CW-structure on $N$ is called a \emph{fundamental family}
and determines a basis of each chain module $\operatorname{C}_k(N)$.

We can tensor the complex 
$\ICellC(N)$ with a $(K(t), \IZ[\pi_1(N)])$-bimodule $A$. 
Here $K(t)$ denotes the quotient field of the polynomial ring
in one variable.
If the chain complex $\ICellC(N; A):= A \otimes_{\IZ[\pi_1(N)]}\ICellC(N)$ is not acyclic,
then we define $\tau(N;A):=0$.
If the chain complex is acyclic, then its \emph{Reidemeister torsion} $\tau(N;
A)\in K(t)\setminus \{0\}$ is defined, see \cite{Turaev01} for an introduction.
We quickly recall its construction.  As the chain complex $\ICellC(N; A)$ is
acyclic, we obtain exact sequences of the form
\begin{align*}
0 \rightarrow \image \partial_{i+1} &\rightarrow \ICellC_i (N;A) \rightarrow \image \partial_i \rightarrow 0.
\end{align*}
We fix a basis for each $\image \partial_i$.
From the basis of $\image \partial_{i+1}$ and $\image \partial_i$, we obtain a basis
$b_i$ of $\ICellC_i (N;A)$, which we will compare with the basis $c_i$ of $\ICellC_i(N;A)$ given
by a fundamental family and a basis of $A$. We denote the matrix expressing the basis $b_i$ in terms
of $c_i$ by $[b_i / c_i]$. Define the Reidemeister torsion of $\ICellC(N;A)$ to be
\begin{align*}
	\tau(N;A) := \prod_i \left(\det [b_{i} /c_i ]\right)^{{(-1)}^i} \in K(t).
\end{align*}
It is a unit in $K(t)$ well-defined up to multiplication with $\pm \det g_A$
for $g \in \pi_1(N)$. 

For future reference we mention the following elementary lemma.

\begin{lem}\label{lem:tapbasics}
If $A$ and $B$ are two $(K(t), \IZ[\pi_1(N)])$-bimodules, then
\[ \tau(N;A\oplus B)=\tau(N;A)\cdot \tau(N;B).\]
\end{lem}

\begin{defn}\label{def:universaltwist}
\begin{enumerate}
\item A representation $V$ of $\pi_1(N)$ \emph{detects the Thurston norm} of $\theta \in \cHo 1 {N;\IZ}$ if
\begin{align*}
\left( \dim V \right)\cdot \Thur \theta = \width \tau(N; V_\theta).
\end{align*}
\item Let $\pi \leq \pi_1(N)$ be a subgroup of finite index and
$V$ a representation of $\pi$. The \emph{induced} representation of $\pi_1(N)$ is
\begin{align*}
\ind{\pi_1(N)} V := V \otimes_{\IZ[\pi]}\IZ[\pi_1(N)].
\end{align*}
Analogously define the induced $(K(t),\IZ[\pi_1(N)])$-bimodule $\ind{\pi_1(N)} A$ 
of a $(K(t),\IZ[\pi])$-bimodule A.
\end{enumerate}
\end{defn}
Gabai \cite[Corollary 6.13]{Gabai83} proved that the Thurston norm is well-behaved under
finite covers. Therefore we are free to consider simpler finite covers. This is
made precise in the lemma below.

\begin{lem}\label{lem:cover}
Let $\Map {p} {M} {N}$ be a connected finite cover and $\theta \in \cHo{1}{N;\IZ}$ a 
cohomology class.
If $V$ detects the Thurston norm of $p^* \theta$, then $\ind{\pi_1(N)} V$ detects
the Thurston norm of $\theta$.
\end{lem}

\begin{proof}
By a result of Gabai\cite[Corollary 6.13]{Gabai83} the Thurston norm fulfils the equality
$\deg p \Thur \theta = \Thur {p^* \theta}$. Furthermore, note that
\begin{align*}
\left(\ind{\pi_1(N)} V\right)_{\theta} = K(t) \otimes_K \left(V \otimes_{\IZ[\pi]}\IZ[\pi_1(N)]\right)
= \ind{\pi_1(N)} \left( V_{p^*\theta}\right)
\end{align*}
as $(K(t), \IZ[\pi_1(N)])$-bimodules. Let us abbreviate $V_{p^*\theta}$ with $A$.

First we prove that $\ICellC(M; A)$ is acyclic if and only if $\ICellC(N; \ind{\pi_1(N)}A )$ is acyclic.
Choose the CW-structure on $M$ which is induced of $N$. The following map 
is an isomorphism of chain complex of $K(t)$-vector spaces
\begin{align*}
A \otimes \IZ[\pi_1(N)] \otimes \ICellC(N) &\rightarrow A \otimes \ICellC(M) \\
v \otimes g \otimes e &\mapsto v \otimes g\cdot e.
\end{align*}
Therefore one is acyclic if and only if the other is.

So let $V$ detect the Thurston norm of $p^*\theta \in \cHo 1 {M;\IZ}$, i.e.  we have
\begin{align*}
\left( \dim V\right)\cdot \Thur {p^*\theta} = \width \tau(M; A).
\end{align*}
Note that we have the equality $\deg p \cdot \dim V = \dim \ind{\pi_1(N)} V$. 

We claim that equality $\width \tau(M; A) = \width \tau(N, \ind{\pi_1(N)}A)$
holds. We pick representatives $g_i$ of right cosets, so $\pi_1(N) = \bigcup
\pi_1(M) \cdot g_i$. Then we equip $\ind{\pi_1(N)}A$ with the basis $\{v
\otimes g_i\}$, where  $\{v\}$ is a basis of $A$. 
Given a fundamental family $\{ \widetilde e\}$
for $N$, we equip $M$ with the fundamental family $\{ g_i \cdot \widetilde e\}$.
With these choices made, the isomorphism above also preserves the basis used for calculating the Reidemeister torsion.
Thus even $\tau(M; A)$ and $\tau(N; \ind{\pi_1(N)}A)$ agree.

Combining the results obtained so far, we get
\begin{align*}
\left( \dim \ind {\pi_1(N)} V \right)\cdot \Thur {\theta} &= \left( \dim V \right) \Thur {p^*\theta}
= \width \tau(M, A)\\
&= \width \tau(N,\ind {\pi_1(N)} A) \\
&= \width \tau(N, (\ind {\pi_1(N)} V)_\theta).
\end{align*}
\end{proof}

\begin{defn}\label{def:proprep}
\begin{enumerate}
\item Given a representation $W$ of the group $H$ and a group homomorphism $G \rightarrow H$,
we can let $\IZ[G]$ act through $\IZ[H]$ and obtain a representation $\res {G} V$ of $G$.
If $V$ is isomorphic to the restriction of a representation of a finite group, we say
$V$ \emph{factors through a finite group}.
\item A representation $V$ of $G$ over $\IC$ is called \emph{integral} if there
is a $(\IZ, \IZ[G])$-bimodule $W$ such that $V \cong \IC\otimes_\IZ W$.
\end{enumerate}
\end{defn}

\begin{rem}\label{rem:inducedreps}
Let $\pi \leq \pi_1(N)$ be a finite index subgroup and $V$ be a
representation of $\pi$. If $V$ is integral, then also $\ind {\pi_1(N)} V$ is integral.
If $V$ factors through a finite group, then also $\ind {\pi_1(N)} V$ factors
through a finite group.
\end{rem}
\begin{defn}\label{def:complexrep}
For a character $\Map \alpha {\pi_1(N)} \IZk$ define
the representation $\IC^\alpha$ to be the representation with underlying
$\IC$-vector space $\IC$ and right action given by
\begin{align*}
\IC^\alpha \times \pi_1(N) &\rightarrow \IC^\alpha\\
(z,g) &\mapsto z\alpha(g),
\end{align*}
where we consider $\IZk$ embedded in $\IC$ as the $k$ roots of unity via 
$n\IZ \mapsto \exp{\left(2\pi i n\right)}$.
\end{defn}
\begin{rem}
The representation $\IC^\alpha$ factors through a finite group but
is in general not integral.
\end{rem}
\section{Graph manifolds}\label{sec:graph}
Recall that every irreducible $3$-manifold $N$ admits the JSJ-decomposition,
a minimal collection $\mathcal{T}$ of embedded incompressible tori such that every component of 
$N | \mathcal{T}$ is either ateroidal or Seifert fibred, where $N | \mathcal{T}$ is the
manifold $N$ split along the tori.
\begin{defn}
An irreducible $3$-manifold $N$ is called a \emph{graph manifold}
if all the pieces of its JSJ-decomposition are Seifert fibred.
\end{defn}
We give a list of examples of graph manifolds below. The list is comprehensive
in the sense that every graph manifold is finitely covered by a manifold
contained in the list, see e.g. \cite[Proposition 2.9]{Nagel14}.
\begin{exm}
\begin{enumerate}
\item the $3$-sphere $S^3$,
\item torus bundles,
\item circle bundles, 
\item pieces of the form $\Sigma \times S^1$ glued together along their boundary tori,
where the surface $\Sigma$ always has negative Euler characteristic.
\end{enumerate}
\end{exm}
The Thurston norm in the first two examples vanishes. 
It also has a simple description for circle bundles and so 
we are mainly interested in understanding the last class
of the list above. We give a more precise description of 
these manifolds as manifolds with graph structures.
\begin{defn}
\begin{enumerate}
\item A \emph{graph structure} for $N$ consists of maps
\begin{align*}
\phi_+ \colon \coprod_{v\in I^+} \Sigma_v \times S^1 \rightarrow N\\
\phi_- \colon \coprod_{v \in I^-} \Sigma_v \times S^1 \rightarrow N
\end{align*}
such that $N$ is the push-out of the following diagram
\begin{center}
	\begin{tikzpicture}
		\matrix (m) [matrix of math nodes, row sep=3em, column sep=4em, 
				text height=1.5ex, text depth=0.25ex]
  		{
		     \coprod_{v\in I^+} \Sigma_v \times S^1  & N \\
		     \coprod T_e & \coprod_{v \in I^-} \Sigma_v \times S^1\\
		};
		\path[->] (m-1-1) edge node [above] {$\phi_+$} (m-1-2);
		\path[->] (m-2-1) edge node [left] {$i^+$} (m-1-1);
		\path[->] (m-2-1) edge node [below] {$i^-$} (m-2-2);
		\path[->] (m-2-2) edge node [right] {$\phi_-$} (m-1-2);
	\end{tikzpicture}
\end{center}
where $i_\pm$ are identifications of $\coprod T_e$ with components of 
$\coprod \partial \Sigma_v \times S^1$. We denote by $\phi_v$ the concatenation
\begin{align*}
\phi_v \colon \Sigma_v \times S^1 \hookrightarrow \coprod_{v\in I^\pm} \Sigma_v \times S^1 \rightarrow N.
\end{align*}

\item For a manifold $N$ with a graph structure, we refer to the classes
$t_v := {\phi_v}_* [*_v \times S^1]$ as the \emph{class of the Seifert fibre} in the block $v$.
Define a character $\Map{\alpha}{\pi_1(N)}{\IZk}$ to be \emph{Seifert non-vanishing} if
$\langle  \alpha, t_v\rangle \neq 0$ for all $v \in I_\pm$.
\end{enumerate}
\end{defn}
\begin{lem}\label{lem:graphcover}
Let $N$ be a graph manifold which does not admit a Seifert fibred structure and
is not a torus bundle.
Then there is a finite cover $\widetilde N$ of $N$ with a graph structure and
on $\widetilde N$ there is for all but finitely many prime numbers $p$ a character
$\Map{\alpha}{\pi_1(\widetilde N)}{\IZp}$ which is Seifert non-vanishing.
\end{lem}
\begin{proof}
The manifold $N$ admits a finite cover $M$ with a collection $\mathcal{T}$ of embedded incompressible tori 
such that each component of $M |\mathcal{T}$ 
is of the form $\Sigma \times S^1$ with $\Sigma$ of the negative Euler
characteristic, see \cite[Section 2]{Nagel14}. To such a decomposition is associated the Bass-Serre
graph which has vertices the components of $M |\mathcal{T}$ and
edges the tori in $\mathcal{T}$. 

The Bass-Serre graph can be chosen to be bipartite.
This can be achieved by a further finite cover which is induced by the kernel 
of the map
\begin{align*}
\pi_1(M) &\rightarrow \IZ_2\\
\gamma &\mapsto \sum_{T \in \mathcal{T}} \gamma \cdot [T].
\end{align*}

The existence of the character $\alpha$ follows from \cite[Theorem 2.15]{Nagel14}.
\end{proof}
Now we construct representations $V$ which will detect
the Thurston norm on a $3$-manifold $N$ with a graph structure
and a Seifert non-vanishing character $\alpha \colon \pi_1(N) \rightarrow \IZp$.
\begin{defn}
\begin{enumerate}
\item A representation $V$ of $\IZp$ is called \emph{good} if
$\Map {(1-g)_V} {V} {V}$ is invertible as an endomorphism of the $K$-vector space $V$
for all non-trivial $g \in \IZp$.
\item A representation $V$ of $\pi_1(N)$ is called \emph{good}
if $V$ is isomorphic to $\res \alpha W$ for a good representation $W$
of $\IZp$ and a Seifert non-vanishing character $\alpha$.
\end{enumerate}
\end{defn}

We give some examples of good representations for $\pi_1(N)$ in
the list below. Recall that the representation $\IC^\alpha$ was defined
in Definition \ref{def:complexrep}.
\begin{exm}
Let $p,q$ be two different prime numbers with $q > 2$ and $\Map \alpha {\pi_1(N)} \IZfZ p$ 
be a Seifert non-vanishing character.
\begin{enumerate}
\item The representation $\IC^\alpha$ is good.
\item The augmentation ideal $I(\IC)$ of $\IC[\IZfZ p]$ is the kernel of the map
\begin{align*}
\IC[\IZfZ p] &\rightarrow \IC\\
\sum_{g \in \IZfZ p} a_g g &\mapsto \sum_{g \in \IZfZ p} a_g.
\end{align*}
It is a good representation of $\IZfZ p$. Thus $I^\alpha_{\infty} := \res \alpha I(\IC)$
is a good representation of $\pi_1(N)$. Additionally, this representation
is integral.
\item The augmentation ideal $I(\IF q)$ of $\IF q [\IZfZ p]$
is good and so is $I_q^\alpha := \res \alpha I(\IF q)$.
\end{enumerate}
\end{exm}
The following theorem is the main reason for considering good representations.
\begin{thm}\label{thm:graphdetecting}
Let $N$ be a graph manifold with a graph structure.
Every good representation $V$ detects the Thurston norm of 
$\theta$ for all $\theta \in \cHo 1 {N; \IZ}$.
\end{thm}
Some of the calculations have already been discussed elsewhere. We only sketch these and refer
to them\cite[Lemma 4.17, Proposition 4.18]{Nagel14}.
\begin{proof}
Pick CW-structures such that the maps $\phi_\pm$ and $i_\pm$ are inclusion
of subcomplexes. Let $\{T_e\}$ be the collection of graph tori.
The graph structure of $N$ gives rise to a short exact sequence
\begin{align*}
0 \rightarrow \bigoplus_e \ICellC(T_e \subset N; V_\theta)
\xrightarrow{i^+ - i^-} \bigoplus_{v \in I_\pm}\ICellC(\Sigma_v \times S^1 \subset N; V_\theta)
\rightarrow &\ICellC(N; V_\theta) \rightarrow 0.
\end{align*}
As the representation $V$ is good, the map $(1-t_v)_V\colon V \rightarrow V$ is invertible.
Thus $\det(1-t_v) _V\neq 0$ and as a consequence $\det (1-t_v)_{V_\theta} \neq 0$.
Therefore the chain complex $\ICellC(\Sigma_v \times S^1 \subset N, V_\theta)$ is acyclic
and its Reidemeister torsion is 
\begin{align*}
\tau\left(\Sigma_v \times S^1 \subset N, V_\theta\right) &= {\det}\left(\left(1-t_v\right)_{V_\theta}\right)^{\chi(\Sigma_v)}.
\end{align*}
This can be seen as follows.
Because $\phi_v$ injects $\pi_1(\Sigma_v \times S^1)$ into $\pi_1(N)$,
we can identify the chain complexes
$\ICellC(\Sigma_v \times S^1 \subset N; V_\theta) \cong
V_\theta \otimes \IC[\pi_1(N)] \otimes \ICellC(\Sigma_v \times S^1)$.
Now the Reidemeister torsion of the right hand side can be calculated explicitly.

Similarly, we prove that the chain complex 
$\ICellC(T_e \subset N; V_\theta)$ is acyclic as well and
$\tau(T_e \subset N; V_\theta) = 1$. We obtain by the above short exact sequence that
\begin{align*}
\tau(N; V_\theta) = \prod_{v \in I_\pm} {\det}\left( (1 - t_v)_{V_\theta}\right)^{\chi(\Sigma_v)}.
\end{align*}
We calculate $\width {\det}_{V_\theta} \left(1 - t_v\right) = \dim V |\langle \theta, [\Sigma_v] \rangle|$.
Thus taking width in the equation above, we get the equalities
\begin{align*}
\width \tau(N, V_\theta) &=\left( \dim V \right) \sum_{v \in I_\pm} \chi(\Sigma_v) |\langle \theta, [\Sigma_v] \rangle|\\
&= \left( \dim V \right)\sum_{v \in I_\pm} \Thur{ {\phi_v}^* \theta } = \left( \dim V \right)\Thur{ \theta }.
\end{align*}
The second equality is a calculation of the Thurston norm in $\Sigma_v \times
S^1$, see e.g. \cite[Proposition 3.4]{Nagel14}.
The last equality holds by a result of Eisenbud-Neumann \cite[Proposition 3.5]{Eisenbud85}.
\end{proof}

For all graph manifolds but $D^2 \times S^1$, we obtain the following theorem. This is precisely the statement of Theorem~\ref{thm:mainthm} for graph manifolds.

\begin{thm}
Let $N$ be a graph manifold which is not $D^2 \times S^1$. 
For every $\theta \in \cHo{1}{N; \IZ}$
there is a representation $V$ factoring through a finite group 
which detects the Thurston norm of $\theta$. Additionally, the
representation $V$ can be chosen to be either
\begin{enumerate}
\item defined over the complex numbers and be integral, or
\item defined over a finite field $\IF q$ with $q > 2$ prime.
\end{enumerate}
\end{thm}

\begin{proof}
The statement is vacuous if $\Thur{\theta}=0$. We can thus restrict ourselves to the case that $\Thur{\theta}>0$. 
In particular, we can assume that $N$ is neither covered by $S^3$ nor is it
covered by a torus bundle. Also, the following claim shows that we can assume
that $N$ is not covered by a non-trivial circle bundle of a surface.

\begin{claim}
If $N$ is covered by a non-trivial circle bundle $\pi \colon E \rightarrow B$, then the Thurston norm vanishes on $N$. 
\end{claim}

Since $\pi$ is a non-trivial circle bundle  the Euler class
is non-trivial. Using the Gysin sequence we see that 
\begin{align*}
H^1(B) \overset{\pi^*}{\rightarrow} H^1(E)
\end{align*}
is surjective. Thus we can represent every class in $H_2(E)$ by 
a multiple of the fundamental class of a torus. Therefore the Thurston norm
vanishes on $E$ and so by \cite[Corollary 6.13]{Gabai83}  also on $N$.  This concludes the proof of the claim.

If $N$ is the trivial circle bundle, then the representation $\IC^\alpha$,
$I_\infty^\alpha$ and  $I_q^\alpha$  will detect the Thurston norm of $\theta$ for any
character $\alpha \colon \pi_1(\Sigma \times S^1) \rightarrow \IZp$ with prime $p$ different from $q$ which
is non-zero on $\pi_1(S^1) \subset \pi_1(\Sigma \times S^1)$. If $N$ is finitely covered
by a trivial circle bundle, then the induced representations 
$\ind {\pi_1(N)} \IC^\alpha$,
$\ind {\pi_1(N)} I_\infty^\alpha$,
and $\ind {\pi_1(N)} I_q^\alpha$ will detect the Thurston norm by Lemma \ref{lem:cover}.

So we are in the case that $N$ does not admit a Seifert fibred structure. 
By Lemma \ref{lem:graphcover} there is a finite cover $M \rightarrow N$ such that $M$ admits a graph structure
and a character $\alpha \colon \pi_1(N) \rightarrow \IZp$ which is Seifert fibre non-vanishing
and with prime $p$ being different from $q$. By Theorem \ref{thm:graphdetecting}
the representations $\IC^\alpha$, $I_\infty^\alpha$ and $I_q^\alpha$ detect the Thurston norm of $p^*\theta$.
By Lemma \ref{lem:cover} the representations $\ind {\pi_1(N)} \IC^\alpha$,
$\ind {\pi_1(N)} I_\infty^\alpha$
and $\ind {\pi_1(N)} I_q^\alpha$ detect the Thurston norm of $\theta$ on $N$.

The representation $\ind {\pi_1(N)} I_\infty^\alpha$ is integral
and $\ind {\pi_1(N)} I_q^\alpha$ is defined over $\IF q$.
\end{proof}

\section{The proof of Theorem~\ref{thm:mainthm} for non-graph manifolds}\label{sec:fibred}

The goal of this section is to prove  Theorem~\ref{thm:mainthm} for 3-manifolds that are not graph manifolds. More precisely, we will prove the following theorem.

\begin{thm}\label{thm:mainthm2}
Let $N$ be an irreducible 3-manifold that is not a  graph manifold and let $\theta \in \cHo{1}{N; \IZ}$.
Then the following hold:
\begin{enumerate}
\item there is an integral representation $V$ factoring through a finite group which detects the Thurston norm, and
\item for almost all primes $q$ there is a representation $V$
  over the finite field $\IF q$ which detects the Thurston norm.
\end{enumerate}
\end{thm}

The proof of Theorem~\ref{thm:mainthm2} is, perhaps not surprisingly, a modification of the proof of the main theorem of \cite{Friedl12}. In an attempt to keep the paper concise we will only indicate which steps of \cite{Friedl12} need to be modified.

A $3$-manifold $N$ is called \emph{fibred} if it can be given the structure of
a surface bundle over $S^1$. We say it is \emph{virtually fibred} if $N$ admits a
finite cover which fibres.
\begin{defn}
A class $\theta \in \cHo{1}{N;\IQ}$ is called \emph{fibred} if
there exists a map $\Map{\rho}{N}{S^1}$ and a class 
$\tau \in \cHo{1}{S^1;\IQ}$ such that $\Map{\rho}{N}{S^1}$
is a fibre bundle and $\rho^* \tau = \theta$.
\end{defn}

The following theorem is the key topological  ingredient in the proof of Theorem~\ref{thm:mainthm2}. This theorem is a combination of the  results of Agol \cite{Agol08,Agol13},
 Przytycki-Wise\cite{PWise12} and Wise \cite{Wise12}.
 We refer to \cite{AFW15} for precise references.

\begin{thm}\label{thm:virtfib}
Let $N$ be an irreducible  $3$-manifold that is not a  graph manifold.
Then given a class $\theta \in \cHo{1}{N;\IQ}$ there exists a finite regular cover $p$
of $N$ such that the class $p^* \theta$ is in the closure of fibred 
classes.
\end{thm}

In the proof of Theorem~\ref{thm:mainthm2} we will need the following lemma, 
which is a slight generalisation of  \cite[Lemma~5.7]{Friedl12}. 

\begin{lem}\label{lem:deltaq}
Let $F$ be a free Abelian group and let $p_1,\dots,p_l\in \IZ[F]$ be  non-zero elements. Then  there exists a prime $q$ and a homomorphism
$\alpha\colon F\to \IZ/q\IZ$ such that for any $j\in \{1,\dots,q-1\}$ the character 
\[ \begin{array}{rcl}  \rho_j\colon  \IZ/q\IZ&\to & S^1 \\
a&\mapsto & e^{2\pi iaj/q}\end{array} \]
has the property that $(\rho_j\circ \alpha)(p_1),\dots,(\rho_{j}\circ \alpha)(p_l)$
 are non-zero elements of $\IC$.
\end{lem}

\begin{proof} 
As in the proof of \cite[Lemma~5.7]{Friedl12} we first note that 
there exists a homomorphism $\Psi\colon F\to \IZ=\langle s\rangle$ such that $\Psi(p_1),\dots,\Psi(p_l)\in \IZ[\IZ]=\IZ[s^{\pm 1}]$ are non-zero polynomials.
Since the polynomials $\Psi(p_1),\dots,\Psi(p_l)$ have finitely many zeros it follows that there exists a prime $q$ such that no primitive $q$-th root of unity is a zero of any $\Psi(p_i)$, $i=1,\dots,l$. 
The homomorphism $F\xrightarrow{\Psi}\IZ\to \IZ/q\IZ$  has the desired property.
\end{proof}

A \emph{representation $V$ of a group $G$ over a ring $R$} is a $(R,\IZ[G])$-bimodule such that $V$ is a  finitely generated free $R$-module.
Given a representation $V$ of $G$ over $\IZ$ we denote by $V^{\IC}=\IC \otimes_{\IZ} V$ 
the corresponding complex representation of $G$ and given a
prime $p$ we denote by $V^p=\IF p \otimes_{\IZ} V$ the corresponding
representation of $G$ over the finite field $\IF p$. 

\begin{lem}\label{lem:taumodp}
Let $N$ be $3$-manifold, let $\theta\in H^1(N;\IZ)$ be non-trivial and let $V$ be an integral representation of $\pi_1(N)$.  Then for  all but finitely many primes $p$ we  have 
\[ \width( \tau(N,V^p_\theta))=\width( \tau(N,V^{\IC}_\theta)).\]
\end{lem}

\begin{proof}
We provide the proof in the case that $N$ is closed. The case that $N$ has non-empty boundary is proved completely analogously.

We write $\pi=\pi_1(N)$.
It follows from \cite[p.~49]{FV10} that there exist $g,h\in \pi$ with
$\theta(g)\ne 0$ and $\theta(h)\ne 0$ and a  square matrix $B$ over $\IZ[\pi]$
such that  for any representation $W$ of $\pi$ over a commutative ring $R$ we
have
\[ \tau(N,W_\theta)=\det_W(B)\cdot \det_W(1-g)^{-1}\cdot \det_W(1-h)^{-1}.\]

Here, given a $k\times k$-matrix $C$ over $\IZ[\pi]$ we denote by $\det_W(C)$ the determinant of the homomorphism
\[  W\otimes_{\IZ[\pi]}\IZ[\pi]^k\to  W\otimes_{\IZ[\pi]}\IZ[\pi]^k\]
given by right multiplication by $\operatorname{id}\otimes C$. 

Now let $V$ be an integral representation of $\pi_1(N)$.  
Given a prime $p$ we denote by $\rho_p\colon \IZ\to \IF p$ the projection map. This map induces a map $\rho_p\colon\IZ[t^{\pm1 }]\to \IF p[t^{\pm 1}]$ and also a map
\[ \rho_p\colon \{ \textstyle \frac{x(t)}{y(t)}\,|\, x(t),y(t)\in \IZ[t^{\pm 1}]\mbox{ with }\rho_p(y(t))\ne 0\}\to \IF p(t).\]
We let $x(t)=\det_V(B)$ and $y(t)=\det_V(1-g)\cdot \det_V(1-h)$. It follows   easily from the above formula for Reidemeister torsions and the fact that taking determinants commutes with ring homomorphisms that
 \[ \tau(N,V_\theta^{\IC})=x(t)y(t)^{-1}\]
 and  that for any prime $p$  we have
\[ \tau(N,V_\theta^p)=\rho_p(x(t))\cdot \rho_p(y(t))^{-1}.\]
Thus, if $p$ is coprime to the bottom and the top coefficients of $x(t)$ and $y(t)$ we  have 
\[ \width( \tau(N,V_\theta^p))=\width( \tau(N,V_\theta^{\IC})).\]
\end{proof}

For the record we recall the well-known elementary fact.

\begin{lem}\label{lem:augmentionideal}
Given $q\in \IN$ the augmentation ideal $I(\IC)$ of $\IC[\IZ/q\IZ]$, viewed as a representation of $\IZ/q\IZ$, is isomorphic to the direct sum of the representations
\[ \begin{array}{rcl} \rho_j\colon \IZ/q\IZ&\to & \operatorname{Aut}(\IC)\\
a&\mapsto & (e^{2\pi ija/q})\end{array}\]
with $j\in \{1,\dots,q-1\}$.
\end{lem}

Now we are in a position to prove Theorem \ref{thm:mainthm2}.

\begin{proof}[Proof of Theorem \ref{thm:mainthm2}]
Let $N$ be $3$-manifold which is not a closed graph manifold and let $\theta\in H^1(N;\IZ)$.
By Theorem~\ref{thm:virtfib}  there exists a finite $k$-fold regular cover $p\colon M\to N$ such that the class $p^* \theta$ is in the closure of fibred 
classes.

If we follow the  proofs of Proposition~5.8 and Theorem~5.9 in \cite{Friedl12},
if we replace \cite[Lemma~5.7]{Friedl12}
by Lemma \ref{lem:deltaq}, and if we apply Lemmas~ \ref{lem:tapbasics} and~\ref{lem:augmentionideal}
then we see that there exists a prime $q$ and a homomorphism
$\alpha\colon \pi_1(M)\to \IZ/q\IZ$ such that 
\[ \width\left(\tau(M,\left(\res {\alpha} I(\IC)\right)_{p^*\theta})\right)=(q-1)x_M(p^*\theta).\]
Put differently, the $(q-1)$-dimensional representation $\res {\alpha} I(\IC)$ detects the Thurston norm of $p^*\theta$.
Now the first part of  theorem is  an immediate consequence of Remark~\ref{rem:inducedreps} together with Lemma~\ref{lem:cover}.

The second part is a consequence of the first part together with Lemma~\ref{lem:taumodp}.
\end{proof}

\bibliography{biblio}{}
\bibliographystyle{alpha}
\end{document}